\documentclass[11pt]{article}
\usepackage[margin=1.1in]{geometry}
\usepackage{amsmath,amssymb,amsthm}
\usepackage[colorlinks=true,linkcolor=blue,citecolor=blue]{hyperref}

\newtheorem{theorem}{Theorem}
\newtheorem{corollary}[theorem]{Corollary}
\newtheorem{proposition}[theorem]{Proposition}
\newtheorem{lemma}[theorem]{Lemma}
\newtheorem{remark}[theorem]{Remark}

\newcommand{\N}{\mathbb{N}}
\newcommand{\R}{\mathbb{R}}

\newcommand{\Z}{\mathbb{Z}}
\newcommand{\dist}{\operatorname{dist}}

\title{A Borsuk-Ulam theorem for adaptive continuous measurements}
\author{Mario Ullrich\\[1mm] \small Johannes Kepler University Linz\\
\small\texttt{mario.ullrich@jku.at}}
\date{October 24, 2026}

\begin{document}
\maketitle

\begin{abstract}
We prove a Borsuk-Ulam theorem for adaptive measurements: if~$2^n\le m$, then any~$n$ adaptively chosen continuous functionals take identical values at two antipodal points of any sphere in~$\R^m$. 
Consequently, at least $\lceil\log_2(m+1)\rceil$ adaptive
continuous measurements are needed to learn the vectors of a ball in~$\R^m$ with an
error smaller than its radius, which matches the recently obtained upper bound up to
one measurement. The bound extends to Bernstein widths in Banach spaces. The main
idea of the proof was found by Claude~Fable~5, an AI model developed by Anthropic,
and we document the development in a last section.
\end{abstract}


\section{The problem and the result}\label{sec:intro}

Fix a norm~$\|\cdot\|$ on~$\R^m$ and a set $F\subseteq\R^m$ of vectors we want to
approximate with a small error that is measured in the given norm. For~$R>0$, we write
\[
S_R^m\;=\;\{x\in\R^m:\ \|x\|=R\}
\]
for the sphere of radius~$R$ in this norm. An
\emph{adaptive continuous measurement map of length~$n$} on~$F$ is built from
\begin{itemize}
\item a continuous functional $\lambda_1\in C(F)$, where~$C(F)$ denotes the
real-valued functions on~$F$ that are continuous with respect to the given norm, and
\item for $2\le k\le n$, an \emph{arbitrary} (no regularity assumed) assignment
$(y_1,\dots,y_{k-1}) \in \R^{k-1} \;\mapsto\; \lambda_k^{(y_1,\dots,y_{k-1})}\in C(F)$:
\end{itemize}
it is the map
\[
N\colon F\to\R^n,\qquad N(x)=(y_1,\dots,y_n),
\]
where $y_1=\lambda_1(x)$ and $y_k=\lambda_k^{(y_1,\dots,y_{k-1})}(x)$ for $2\le k\le
n$. 
We refer to the $\lambda_k$ as \emph{measurements}. 
An \emph{algorithm using~$n$ adaptive continuous measurements} is a map
$A=\Phi\circ N$ with an adaptive continuous measurement map~$N$ of length~$n$ and an
arbitrary reconstruction map $\Phi\colon\R^n\to\R^m$. Its \emph{worst-case error} is
\[
e(A)\;=\;\sup_{x\in F}\,\bigl\|x-\Phi(N(x))\bigr\| .
\]
This is the model of \cite{KNU}; see \cite[Sec.~10]{KU} for a survey and \cite{KNPU}
for noisy measurements. The algorithm of \cite{KNU} learns every~$x\in\R^m$ with
arbitrary precision from $\lceil\log_2(m)\rceil+1$ adaptive continuous measurements.
We show that this is optimal up to one measurement.

The basic principle for lower bounds is \emph{indistinguishability}: if~$N(x)=N(x')$, the algorithm returns the same point for both inputs, so
$2\,e(A)\ge\|x-x'\|$. For a nonadaptive measurement map, which is a continuous map
$N\colon F\to\R^n$, the Borsuk-Ulam theorem provides indistinguishable pairs: every
continuous map~$S_R^m\to\R^n$ with~$n\le m-1$ takes identical
values at some pair of antipodal points \cite{Bor,Mat}. An adaptive measurement map
is in general \emph{discontinuous}, since $\lambda_k^{(\cdot)}$ may depend on the
earlier values in a completely irregular way, and the classical theorem says nothing
about it. The following is the statement that survives. We say that~$N$
\emph{distinguishes antipodes} on $S_R^m\subseteq F$ if
$N(x)\neq N(-x)$ for all~$x\in S_R^m$.

\begin{theorem}[Borsuk-Ulam for adaptive measurements]\label{thm:main}
Let~$\|\cdot\|$ be any norm on~$\R^m$, let~$R>0$, and let~$N$ be an adaptive
continuous measurement map of length~$n$ on a set containing~$S_R^m$. If~$2^n\le m$,
then there is~$x\in S_R^m$ with
\[
N(x)\;=\;N(-x) .
\]
Equivalently, if~$N$ distinguishes antipodes on~$S_R^m$, then~$2^n\ge m+1$.
\end{theorem}

Adaptivity thus moves the Borsuk-Ulam threshold from~$m-1$ to~$\log_2 m$, and no
further: for $n\ge\lceil\log_2(m)\rceil+1$, the algorithm of \cite[Thm.~1]{KNU}
with precision $\varepsilon<R$ distinguishes antipodes on~$S_R^m$.

\medskip
By indistinguishability, Theorem~\ref{thm:main} says that, if~$2^n\le m$, every
algorithm using~$n$ adaptive continuous measurements on a set $F\supseteq S_R^m$ has
worst-case error~$e(A)\ge R$. If~$F$ is the ball of radius~$R$, this is the error of
the trivial algorithm~$A\equiv0$: below the threshold, the measurements are useless,
while $\lceil\log_2(m)\rceil+1$ of them give arbitrary precision \cite{KNU}.

\begin{corollary}\label{cor:learn}
Let $0<\varepsilon<R$. The minimal number $n^*(m,\varepsilon)$ of adaptive
continuous measurements of an algorithm with worst-case error $\varepsilon$ on a
ball of radius~$R$ in~$\R^m$ satisfies
\[
\lceil\log_2(m+1)\rceil\;\le\; n^*(m,\varepsilon)\;\le\;\lceil\log_2(m)\rceil+1 .
\]
\end{corollary}

The bounds coincide if~$m$ is a power of two, and differ by~$1$ otherwise; see
Remark~\ref{rem:gap}.

\medskip
The proof of Theorem~\ref{thm:main} measures the difficulty of separating antipodal
pairs by the genus of Krasnosel'skii, the minimal number of sign-decidable regions covering the set, and shows that one measurement at most
doubles-plus-one resolves it. In words:
\begin{itemize}
\item The odd part $\delta(x)=\lambda(x)-\lambda(-x)$ of a measurement~$\lambda$
separates~$x$ from~$-x$ wherever $\delta\neq0$; this region has genus at most~$1$.
\item Where~$\delta=0$, the measurement is even and only sorts the points into level
sets, on which the algorithm continues with one measurement less. Level sets of the
same parity can share their genus, so this part costs at most a factor of two.
\item Adaptivity is invisible on a level set: antipodal points there have produced
identical values so far and face the same next functional.
\item Hence~$n$ measurements that distinguish antipodes on a set~$K$ force
$\gamma(K)\le2^n-1$, while the sphere has genus~$m$.
\end{itemize}
Sections~\ref{sec:genus} and~\ref{sec:proof} contain the proof,
Section~\ref{sec:banach} the extension to Banach spaces, and
Section~\ref{sec:history} documents how the proof was found in
dialogue with an AI assistant.

\section{Split sets and the genus}\label{sec:genus}

Throughout,~$K$ denotes a compact \emph{symmetric} ($K=-K$) subset of
$\R^m\setminus\{0\}$, 
equipped with the given norm (for distances and continuity). 
A map~$\varphi$ on~$K$ is \emph{odd} if $\varphi(-x)=-\varphi(x)$ and \emph{even} if
$\varphi(-x)=\varphi(x)$. A closed symmetric set $A\subseteq K$ is \emph{split} if
\[
A=A^{+}\cup A^{-},\qquad A^{\pm}\ \text{closed},\qquad A^{+}\cap A^{-}=\varnothing,
\qquad -A^{+}=A^{-} ,
\]
i.e., if a sign distinguishing~$x$ from~$-x$ can be assigned consistently on~$A$.
The \emph{genus} of~$K$ is the minimal number of split sets needed to cover~$K$,
\[
\gamma(K)\;:=\;\min\bigl\{k\ge0:\ K=A_1\cup\dots\cup A_k\ \text{with split sets}\
A_1,\dots,A_k\bigr\},
\]
so that $\gamma(K)=0$ iff $K=\varnothing$. This is Krasnosel'skii's genus
\cite{Kras,Rab,Str}, usually defined via odd maps:

\begin{lemma}\label{lem:genus}
$\gamma(K)$ is the least~$k\ge0$ such that there is a continuous odd map
$K\to\R^k\setminus\{0\}$.
\end{lemma}

\begin{proof}
Let $K=A_1\cup\dots\cup A_k$ with split sets~$A_i$. The distance functions
\[
\sigma_{A_i}(x)\;=\;\dist(x,A_i^{-})-\dist(x,A_i^{+}),\qquad x\in K,
\]
are continuous, odd (as $\dist(-x,A_i^{-})=\dist(x,A_i^{+})$), and have no zero on~$A_i$; so $(\sigma_{A_1},\dots,\sigma_{A_k})$ is a continuous odd map
$K\to\R^k\setminus\{0\}$. Conversely, let $\psi\colon K\to\R^k\setminus\{0\}$ be
continuous and odd. The sets $A_i=\{x\in K:|\psi_i(x)|=\|\psi(x)\|_\infty\}$ are
closed, symmetric and cover~$K$, and~$\psi_i$ has no zero on~$A_i$, so~$A_i$ is
split into $A_i\cap\{\psi_i>0\}$ and $A_i\cap\{\psi_i<0\}$.\\
\end{proof}

\medskip
Let $A,B\subseteq K$ be closed and symmetric, and let
\[
U_t(A)\;=\;\{x\in K:\ \dist(x,A)\le t\},\qquad t>0,
\]
be the closed $t$-neighborhood of~$A$ in~$K$. We need the following properties.
\begin{enumerate}
\item[(G1)] $\gamma(A)\le\gamma(B)$ if $A\subseteq B$, or if there is a continuous
odd map $f\colon A\to B$: intersect, or pull back by~$f$, a covering of~$B$.
\item[(G2)] $\gamma(A\cup B)\le\gamma(A)+\gamma(B)$: unite the coverings.
\item[(G3)] $\gamma(A\cup B)=\max\{\gamma(A),\gamma(B)\}$ if $A\cap B=\varnothing$:
if~$A_i,B_i$ are split, so is $A_i\cup B_i$, with halves $A_i^\pm\cup B_i^\pm$.
\item[(G4)] There is~$t>0$ with $\gamma(U_t(A))=\gamma(A)$: if $A=A_1\cup\dots\cup A_k$
with split sets~$A_i$ and $0<t<\tfrac12\min_i\dist(A_i^{+},A_i^{-})$, then the sets
$\{\dist(\cdot,A_i^{+})\le t\}\cup\{\dist(\cdot,A_i^{-})\le t\}$ are split and cover~$U_t(A)$.
\item[(G5)] $\gamma(S_R^m)=m$: by Lemma~\ref{lem:genus}, this is the Borsuk-Ulam
theorem \cite{Bor,Mat}, every odd continuous map~$S_R^m\to\R^k$ with~$k<m$ has a
zero; it is usually stated for the Euclidean sphere, which is mapped onto~$S_R^m$
by the odd homeomorphism $x\mapsto Rx/\|x\|$, see (G1).
\end{enumerate}

\section{Proof of Theorem~\ref{thm:main}}\label{sec:proof}

We will show that the existence of an adaptive continuous measurement map that can distinguish antipodes implies a \emph{small} genus. 

\begin{proposition}\label{prop:main}
Let $K\subseteq\R^m\setminus\{0\}$ be compact and symmetric, and let an adaptive
continuous measurement map~$N$ of length~$n$ on~$K$ distinguish antipodes, i.e.\
$N(x)\neq N(-x)$ for all~$x\in K$. Then $\gamma(K)\le2^n-1$.
\end{proposition}

\begin{proof}
Induction on~$n$. For~$n=1$, let $\lambda_1$ be the (first and only) measurement. The function $\delta(x)=\lambda_1(x)-\lambda_1(-x)$ is
odd, continuous and has no zero on~$K$, since $\lambda_1$ distinguishes all antipodes. So,~$K$ is split into $K\cap\{\delta\ge0\}$ and
$K\cap\{\delta\le0\}$: both sets are closed, they are disjoint since~$\delta$ has
no zero, and they are exchanged by $x\mapsto-x$ since~$\delta$ is odd. Hence
$\gamma(K)\le1$.

Let~$n\ge2$ and $h=2^{n-1}-1$. Again $\delta(x)=\lambda_1(x)-\lambda_1(-x)$ is odd
and continuous, and
\[
Z\;=\;\delta^{-1}(0)\;=\;\{x\in K:\lambda_1(x)=\lambda_1(-x)\}
\]
is compact and symmetric. If $Z=\varnothing$, then $\gamma(K)\le1$ as before. So let
$Z\neq\varnothing$. The function $\mu=\lambda_1|_Z$ is \emph{even}, so all level sets
$F_y=\mu^{-1}(y)$ and slabs $\mu^{-1}(I)$, with an interval $I\subseteq\R$, are symmetric.

\emph{Level sets.} Let~$y\in\mu(Z)$. All~$x\in F_y$ have the same first value~$y_1=y$, so the second functional is the same for all points of~$F_y$, and the
remaining measurements form an adaptive continuous measurement map~$N_y$ of length~$n-1$ on~$F_y$ with $N(x)=(y,N_y(x))$. Antipodal pairs in~$F_y$ share the first
value, so~$N_y$ distinguishes antipodes on~$F_y$, and the induction hypothesis gives
$\gamma(F_y)\le h$. By (G4) and (G1) there is~$\delta_y>0$ with $\gamma(U_{\delta_y}(F_y)\cap Z)\le h$. Moreover there is~$\rho_y>0$ with
\[
\mu^{-1}\bigl([\,y-\rho_y,\;y+\rho_y\,]\bigr)\;\subseteq\;U_{\delta_y}(F_y):
\]
otherwise there are $x_j\in\mu^{-1}([y-\tfrac1j,y+\tfrac1j])$ with
$\dist(x_j,F_y)>\delta_y$; since~$Z$ is compact, a subsequence converges to
some~$x^*\in Z$, which satisfies $\mu(x^*)=y$, i.e., $x^*\in F_y$, and
$\dist(x^*,F_y)\ge\delta_y$, a contradiction. By (G1), every slab
$\mu^{-1}(I)$ with $I\subseteq(y-\rho_y,y+\rho_y)$ has genus at most~$h$. For
$y\notin\mu(Z)$ we choose~$\rho_y>0$ such that the slab is empty.

\emph{Slabs.} Let $[c,d]=[\min_Z\mu,\max_Z\mu]$. If~$c=d$, then~$Z=F_c$ and
$\gamma(Z)\le h$. Otherwise, the open intervals $(y-\rho_y,y+\rho_y)$,~$y\in[c,d]$,
cover~$[c,d]$. By compactness of~$[c,d]$, there is a Lebesgue number of this
cover, i.e., some~$\eta>0$ such that every subinterval of~$[c,d]$ of length less
than~$\eta$ is contained in one of the intervals $(y-\rho_y,y+\rho_y)$. Choose a
partition $c=t_0<t_1<\dots<t_M=d$ with widths less than~$\eta$. Then each
$[t_{i-1},t_i]$ lies in some $(y-\rho_y,y+\rho_y)$, so the slab
$S_i=\mu^{-1}([t_{i-1},t_i])$ has genus at most~$h$. Slabs of the same parity
are pairwise disjoint, so by (G3)
\[
\gamma\Bigl(\bigcup_{i\ \mathrm{odd}}S_i\Bigr)\le h,
\qquad
\gamma\Bigl(\bigcup_{i\ \mathrm{even}}S_i\Bigr)\le h,
\]
and (G2) gives $\gamma(Z)\le2h$.

\emph{The rest.} Since~$Z$ is compact and symmetric with $\gamma(Z)\le2h$, (G4)
gives some~$r>0$ with $\gamma(U_r(Z))\le2h$. Let
\[
W\;=\;\{x\in K:\ \dist(x,Z)\ge r\} .
\]
This set is closed and symmetric, the latter since~$Z$ is symmetric and hence
$\dist(-x,Z)=\dist(x,Z)$. As $W\cap Z=\varnothing$ and $Z$ is the zero set of~$\delta$,
the odd function~$\delta$ has no zero on~$W$, so~$W$ is split into
$W\cap\{\delta\ge0\}$ and $W\cap\{\delta\le0\}$ as in the case~$n=1$, and
$\gamma(W)\le1$. Every $x\in K$ has $\dist(x,Z)\le r$ or $\dist(x,Z)\ge r$, so
$K=U_r(Z)\cup W$, and (G2) gives $\gamma(K)\le2h+1=2^n-1$.
\end{proof}

\begin{proof}[Proof of Theorem~\ref{thm:main}]
Let~$N$ distinguish antipodes on~$S_R^m$. Restricting all functionals to~$K=S_R^m$ gives
a measurement map on~$K$ that distinguishes antipodes. By (G5) and
Proposition~\ref{prop:main}, $m=\gamma(S_R^m)\le2^n-1$.\\
\end{proof}

\medskip
\begin{remark}[An odd map built from the measurements]\label{rem:explicit}
If~$N$ distinguishes antipodes on~$S_R^m$, the proof covers~$S_R^m$ by $2^n-1$ split
sets built from the measurements: one from the sign regions $\{\delta\ge0\}$ and
$\{\delta\le0\}$ of~$\delta$ off a neighborhood of its zero set~$Z$, and
$2(2^{n-1}-1)$ inside~$Z$, since~$Z$ is covered by two families of disjoint slabs on
which only $n-1$ measurements remain, so that the count for~$n-1$ enters twice. The
distance functions of Lemma~\ref{lem:genus} turn this covering into an odd map
$S_R^m\to\R^{2^n-1}\setminus\{0\}$. In this sense, Theorem~\ref{thm:main} is the
Borsuk-Ulam theorem applied to an odd map built from the measurements.
\end{remark}

\begin{remark}[The gap for~$m$ not a power of two]\label{rem:gap}
For~$m$ not a power of two, Theorem~\ref{thm:main} and \cite{KNU} leave a gap of one
measurement. The method of this paper cannot close it: every compact symmetric
$K\subseteq\R^m\setminus\{0\}$ has $\gamma(K)\le m$, since the identity is an odd map
without zeros, so bounding the genus never gives more than~$2^n\ge m+1$. M.~Dolbeault and D.~Krieg
suggested to look for indistinguishable pairs on tori instead of spheres, and
Claude~Fable~5 turned this idea into a proof that, for every~$m$,
$\lceil\log_2 m\rceil+1$ adaptive continuous measurements are necessary to learn
the vectors of a ball in~$\R^m$ with sufficiently small error, so that the upper
bound of \cite{KNU} is sharp. That proof is rather involved, and a
transparent proof remains to be found. Whether antipodal coincidences on the sphere
are forced for $n=\lceil\log_2 m\rceil$ is open.
\end{remark}

\section{Subsets of Banach spaces and linear operators}\label{sec:banach}

Let~$X$ and~$Y$ be Banach spaces, $S\colon X\to Y$ a bounded linear operator, and
$F\subseteq X$ a bounded set. We want to approximate~$S(f)\in Y$ for~$f\in F$ from adaptive
continuous measurements defined on~$X$: the functionals $\lambda_k^{(\cdot)}$ are in~$C(X)$, the functions on~$X$ that are continuous in the norm of~$X$, and an algorithm
is $A=\Phi\circ N$ with arbitrary $\Phi\colon\R^n\to Y$. The minimal worst-case error
of algorithms using~$n$ such measurements is
\[
e_n^{\rm cont}(S,F)\;=\;\inf_{A}\ \sup_{f\in F}\,\bigl\|S(f)-A(f)\bigr\|_Y .
\]
Following \cite{KNU24}, the Bernstein numbers of~$S$ on~$F$ are
\[
b_k(S,F)\;=\;\sup_{\substack{V\subseteq X\ \text{affine}\\ \dim V=k+1}}\
\sup_{g\in F\cap V}\ \inf_{f\in V\setminus F}\,\|S(f)-S(g)\|_Y ,
\]
the largest radius of a ball around a point of~$F$ in an affine
$(k+1)$-dimensional subspace that is contained in~$F$, where distances are measured
through~$S$. For~$F=B_X$, the unit ball of~$X$, these are the Bernstein numbers~$b_k(S)$
of~$S$; for~$S$ the identity, these are the Bernstein widths~$b_k(F,X)$ of~$F$, see
\cite{Pinkus}. For unbounded~$F$, one may use the general definition of
\cite{KNU24}; the proof below is the same.

\begin{theorem}\label{thm:banach}
For all~$n\in\N$,
\[
e_n^{\mathrm{cont}}(S,F)\;\ge\;b_{2^n-1}(S,F) .
\]
\end{theorem}

\begin{proof}
We may assume $b_{2^n-1}(S,F)>0$. Let~$V\subseteq X$ be affine with~$\dim V=2^n$
and~$g\in F\cap V$ such that the infimum in the definition of~$b_{2^n-1}(S,F)$ is
positive, and let $0<r<\inf_{f\in V\setminus F}\|S(f)-S(g)\|_Y$. Then
\[
K\;=\;\bigl\{f\in V:\ \|S(f)-S(g)\|_Y=r\bigr\}\;\subseteq\;F .
\]
Since~$F$ is bounded and~$V$ is not, the infimum is positive only if~$S$ is injective
on the linear space~$E=V-g$; so $\|h\|_E:=\|S(h)\|_Y$ is a norm on~$E$, and
$K-g=\{h\in E:\|h\|_E=r\}$ is the sphere of radius~$r$ of a $2^n$-dimensional normed
space. The translated measurements $h\mapsto\lambda_k^{(\cdot)}(g+h)$ form an
adaptive continuous measurement map~$N_g$ on~$K-g$, continuous since all norms on~$E$
are equivalent, with $N_g(h)=N(g+h)$. Identifying~$(E,\|\cdot\|_E)$ with~$\R^{2^n}$
and a norm on it, (G5) gives $\gamma(K-g)=2^n>2^n-1$, so by
Proposition~\ref{prop:main} the map~$N_g$ does not distinguish antipodes: there is
$h\in K-g$ with $N(g+h)=N(g-h)$. Both $g+h$ and $g-h$ lie in~$K\subseteq F$, and the
algorithm returns the same element for them, so
\[
\sup_{f\in F}\,\|S(f)-A(f)\|_Y\;\ge\;\tfrac12\,\bigl\|S(g+h)-S(g-h)\bigr\|_Y
\;=\;\|S(h)\|_Y\;=\;r .
\]
Taking the supremum over~$V$, $g$ and~$r$ proves the claim.\\
\end{proof}

\medskip
For comparison, adaptive algorithms using~$n$ \emph{linear} measurements have
worst-case error at least~$b_n(S,F)$, see \cite[Prop.~10]{KNU24}, by the same
indistinguishability argument with linear algebra in place of the Borsuk-Ulam
theorem. Continuous measurements shift the index from~$n$ to~$2^n-1$; that an
exponential shift is unavoidable follows from the upper bound of \cite{KNU}, see
Corollary~\ref{cor:hilbert}.

\medskip
For~$F=B_X$, this complements the upper bound of \cite{KNU}. Let $e_k^{\rm non}(S)$
be the minimal error of algorithms using~$k$ nonadaptive continuous measurements on~$B_X$. Then $b_k(S)\le e_k^{\rm non}(S)\le\delta_k(S)$, where $\delta_k(S)$ are the
manifold widths, see \cite[Sec.~10]{KU}. Both~$b_k$ and~$\delta_k$ are (up to the
indexing)~$s$-numbers in the sense of Pietsch \cite{KU,Mathe}, and all~$s$-numbers coincide if~$X$ and~$Y$ are Hilbert spaces
\cite{Pietsch}. Hence $e_k^{\rm non}(S)=b_k(S)$ in this case, and
$e_n^{\rm cont}(S,B_X)\le e^{\rm non}_{2^{n-2}}(S)$ for~$n\ge2$ by
\cite[Sec.~4]{KNU}.

\begin{corollary}\label{cor:hilbert}
Let~$X,Y$ be Hilbert spaces and $S\colon X\to Y$ bounded and linear. Then, for~$n\ge2$,
\[
e^{\rm non}_{2^n-1}(S)\;\le\;e_n^{\rm cont}(S,B_X)\;\le\;e^{\rm non}_{2^{n-2}}(S).
\]
\end{corollary}

\medskip
The exponential speed-up of adaption found in \cite{KNU} is thus optimal, up to a
factor~$8$ in the index.

\begin{remark}
Whether $e_n^{\rm cont}(S,B_X)\ge c\,e^{\rm non}_{2^n-1}(S)$ holds, with some
constant~$c>0$, in all Banach spaces is open. The method of this paper cannot decide
it: it only detects pairs $f,-f$ on spheres that are indistinguishable, and these
yield the Bernstein numbers. A
direct comparison would require an adaptive measurement map of length~$n$ to factor
through a continuous nonadaptive one of length~$2^n-1$, which fails in general, since
the fibers of~$N$ inherit the discontinuous choice of~$\lambda_2^{(y_1)}$.
\end{remark}

\newpage
\section{How this paper came about: working with an AI assistant}\label{sec:history}

The mathematics of this paper was found in July 2026 in a dialogue with Claude~Fable~5, an AI model developed by Anthropic, over the course of a few hours spread
over a handful of sessions. Since such collaborations are new to the field, and
since the path from question to proof was itself instructive, we document it here.
All prompts below are quoted verbatim; all proofs were subsequently checked by the
author.

\subsection{History of the problem}

Before \cite{KNU}, information from arbitrary continuous measurements had hardly
been considered: in information-based complexity, admissible information is
traditionally linear \cite{KNU24,No96}, and continuous information appeared mainly
implicitly, as the benchmark behind manifold widths \cite{DHM}. The question of what
adaptively chosen continuous measurements can do arose in discussions of David
Krieg, Erich Novak and the author, who worked on it for about a year, beginning in
Passau in the summer of 2024 and continuing at the Isaac Newton Institute in
Cambridge. The outcome was a surprise: $\lceil\log_2(m)\rceil+1$ such measurements
recover a vector of~$\R^m$ with arbitrary precision, exponentially fewer than the~$m$ linear measurements that are necessary and sufficient in the classical setting.
For the lower bound, only~$n\ge2$ was known, from the Borsuk-Ulam theorem applied to
the first measurement, and it was left open in \cite{KNU} whether more can be said.
It was not even clear to us whether the truth was near~$2$ or near~$\log_2 m$.

\subsection{The first prompt and the first proof}

The author put the open problem to the model with deliberately modest expectations,
suggesting a target strictly below the conjectured truth:

\begin{quote}\small
``prove a lower bound by proving that log(m)/2 is not enough, since there are
elements far away that cannot be distinguished''
\end{quote}

The model returned, within one response, a complete draft of a stronger statement:
antipodal pairs on a sphere survive any $n\le\log_2(m+1)-1$ adaptive continuous
measurements. The mechanism was already that of Section~\ref{sec:proof}, with the
cohomological~$\Z_2$-index \cite{Bartsch,FR,Yang} of the surviving symmetric set in
place of the genus: the odd part of a measurement removes one unit of index, and its
level sets cost a factor of two. The index machinery proves the theorem, but it is
heavy: the draft relied on relative cup products and the continuity of \v{C}ech
cohomology \cite{Bartsch}, which the author could not check himself.

\subsection{Toward a proof to understand}

The first draft closed with a remark suggesting an explicit map, and the author
asked
\begin{quote}\small
``does the statement of the final remark give an easier proof?''
\end{quote}
The answer was a reformulation in the language of Smale and Vassiliev \cite{Sma,Vas},
who measure the complexity of algorithms that branch on real-valued tests by the
topology of the problem. The relevant invariant is the Schwarz genus \cite{Sch} of the
map $x\mapsto\{x,-x\}$ from~$S_R^m$ onto the space of antipodal pairs
$S_R^m/\{x\sim-x\}$: the least number of open sets covering the quotient on each of which one point of every pair $\{x,-x\}$ can be
chosen continuously. In this language,
Theorem~\ref{thm:main} says that $n$ adaptive real-valued queries resolve a double
covering of Schwarz genus at most~$2^n-1$, as if each real-valued query were a
binary test. This was the first version with the final threshold~$2^n\le m$, and the
reason is the invariant: a genus is subadditive, $\gamma(A\cup B)\le\gamma(A)+\gamma(B)$,
whereas the index satisfies only $\operatorname{ind}(A\cup B)\le\operatorname{ind}(A)+\operatorname{ind}(B)+1$ and loses one unit per measurement. Still unfamiliar with these methods, the author
asked for a geometric proof,
\begin{quote}\small
``is there a way to prove it more geometrical. if yes, do as extra section''
\end{quote}
and the model replaced the index by the covering version of Krasnosel'skii's genus of
Section~\ref{sec:genus}, which is the Schwarz genus of the antipodal covering with
closed sets in place of open ones, no difference on compact spaces. Its properties
(G1) to (G4) were proved by covering arguments, so that the Borsuk-Ulam theorem
remained the only topological input. Next came the
original language of Krasnosel'skii,
\begin{quote}\small
``add a proof using the original definition of genus with continuous odd functions.''
\end{quote}
and the proof became purely functional-analytic: odd Tietze extensions and bump
functions in the measured value. This version ended with the remark that the odd
parts of the measurements assemble into an explicit odd map
$S_R^m\to\R^{2^n-1}\setminus\{0\}$. The last request,
\begin{quote}\small
``is there a more direct construction without partition of unity?''
\end{quote}
replaced Tietze extensions and bump functions by split sets and distance functions.
This gave the odd map built from the measurements of Remark~\ref{rem:explicit}, with
the distance functions of Lemma~\ref{lem:genus} as the final ingredient.

\subsection{One theorem, several dialects}

Within a few sessions, the same theorem had thus been proved in four languages: with
the cohomological $\Z_2$-index of algebraic topology; with the Schwarz genus, in
the language of topological complexity; with the genus of critical point theory, in
its covering form; and with odd functions, Tietze extensions and partitions of unity,
in the language of analysis.
Not every version taught us something new, but it was a pleasure to see the same
mechanism in these different forms, so that we could choose the dialect in which the
argument was easiest to check and to communicate. For the presentation we chose the
genus, which has citable properties and an elementary covering proof. We enjoyed
working with the model as a translator between mathematical languages and ideas, and
we find it remarkable that it settled an open problem of our field with easy methods
from another.

The division of labor was consistent throughout: the author chose the questions,
rejected or accepted reformulations, and checked the proofs; the model produced the
mechanisms and the drafts. The drafts were not papers: some arguments were short,
definitions were missing, and notation changed between versions. Turning the
mathematics, found within hours, into the present text took considerably longer,
over many rounds of reading, questioning and rewriting, and in the end a good deal
of editing by hand.

\end{document}